\documentclass{article}
\usepackage[colorlinks,citecolor=blue,linkcolor=blue,urlcolor=blue,filecolor=blue,breaklinks]{hyperref}
\usepackage{pdfpages}
\usepackage[utf8]{inputenc}
\usepackage{amssymb}
\usepackage{amsfonts}
\usepackage{amsmath}
\usepackage{amsthm}
\usepackage{mathtools}
\usepackage{mathrsfs}
\usepackage{cleveref}
\usepackage{relsize}
\usepackage{tikz-cd}
\usepackage{geometry}
\usepackage{graphicx}
\usepackage{booktabs}
\usepackage{array}
\usepackage{paralist}
\usepackage{verbatim}
\usepackage{subfig}
\usepackage{epsfig}
\usepackage{changepage}

\theoremstyle{plain}
\newtheorem{theorem}{Theorem}

\newtheorem{lemma}[theorem]{Lemma}

\newtheorem{proposition}[theorem]{Proposition}

\theoremstyle{definition}
\newtheorem{definition}[theorem]{Definition}

\newtheoremstyle{named}%
    {}{}{\itshape}{}{\bfseries}{.}{.5em}{\thmnote{#3}}
\theoremstyle{named}

\newcommand{\R}{\mathbb{R}}
\newcommand{\Z}{\mathbb{Z}}
\newcommand{\Q}{\mathbb{Q}}
\newcommand{\C}{\mathbb{C}}

\newcommand{\id}{\mathrm{Id}}
\newcommand{\fr}{\mathtt{f}}
\renewenvironment{proof}[1][Proof]{\noindent\textbf{#1.} }{\ \rule{0.5em}{0.5em}\par\addvspace{\baselineskip}}

\newcommand{\Heis}{\mathcal{H}}

\renewcommand{\geq}{\geqslant}
\renewcommand{\leq}{\leqslant}

\definecolor{cycleBM}{RGB}{0,150,0}
\definecolor{cycleD}{RGB}{255,0,0}

\begin{document}
\title{Weakly framed surface configurations, Heisenberg homology and  Mapping Class Group action}
\author{ Awais Shaukat\footnote{Abdus Salam School of Mathematical Sciences, Lahore, Pakistan},Christian Blanchet\footnote{Universit{\'e} Paris Cité and Sorbonne Universit{\'e}, CNRS, IMJ-PRG, F-75013 Paris, France}}

\maketitle
\begin{abstract}
We obtain representations of the {\em $\fr$-based} Mapping Class Group of  oriented punctured surfaces from an action of mapping classes on Heisenberg homologies of a circle bundle over surface configurations.
\\

\noindent
\textbf{2020 MSC}: 57K20, 55R80, 55N25, 20C12, 19C09 \\
\textbf{Key words}: Mapping class group,  configuration spaces,  Heisenberg homology.
\end{abstract}

\section*{Introduction}
It is shown in \cite{HeisenbergHomology} that the braid group of an  oriented surface with one boundary component has a natural quotient isomorphic to the Heisenberg group. From this one obtains homologies with coefficients in any representation of the Heisenberg group. The Mapping Class Group acts on the local coefficients and in general there is a twisted action of the  Mapping Class Groups on Heisenberg homologies. For specific representations, including the famous Shr\"odinger one, the Mapping Class Group action can be untwisted, producing a native representation of a central extension.

In the closed case a similar quotient exists in genus $1$, but produces a version of the Heisenberg group with finite center in higher genus \cite{Causin_p}, or more involved metabelian quotients \cite{An}. Here  we will recover  an homomorphism to the full Heisenberg group by replacing the surface braid group by a central extension realized using an $S^1$-bundle over the configuration space.
Elements of this bundle will be called {\em weakly framed} configurations, and its fundamental group named the  {\em weakly framed braid group}. We obtain a presentation for this newly defined group and a quotient homomorphism to an Heisenberg group with infinite cyclic center. We then define homologies of weakly framed configurations with coefficients in any representation of the Heisenberg group. Finally we construct a twisted  action of what we call the  {\em $\fr$-based} Mapping Class Group, a central extension of the Mapping Class Group of the punctured surface whose elements are represented by diffeomorphisms fixing the {\em weakly framed} set of punctures.

The Heisenberg group in genus $g$ is usually realized as a group of $(g+2)\times(g+2)$ matrices. In  Section \ref{regular} we introduce the {\em linearised regular representation} which achieves it as a group of $(2g+2)\times(2g+2)$ matrices.
Using this representation as local coefficients, we obtain a native representation (no twisting) of the {\em $\fr$-based}  Mapping Class Group.

 A famous result of Bigelow~\cite{Bigelow2001} and Krammer~\cite{Krammer2002} states that the classical braid
 groups which are Mapping Class Groups in genus zero are linear. Bigelow's proof uses an homological action on the $2$-points configuration space in the punctured disc.  Then Bigelow and Budney~\cite{Bigelow-Budney} deduced that the mapping class group of the closed orientable surface of genus $2$ is also linear. We speculate that our representations can be used for the linearity problem in higher genus.

\paragraph{Acknowledgements.} We are thankful for the support of the Abdus Salam School of Mathematical Sciences.
This paper is part of the PhD thesis of the first author. We are grateful to Martin Palmer for very useful comments on the preliminary version of this paper.

\section{Weakly framed configurations}
Let $\Sigma_g$, $g\geq 1$, be a closed oriented genus $g$ surface. For $n\geq 2$, the unordered configuration space of $n$ points in $\Sigma_g$ is
\[
\mathcal{C}_{n}(\Sigma_g )= \{ \{c_{1},\dots,c_{n}\} \subset \Sigma_g \mid c_i\neq c_j \text{ for $i\neq j$}\}.
\]
 The surface braid group  is then defined as $\mathbb{B}_{n}(\Sigma_g)=\pi_{1}(\mathcal{C}_{n}(\Sigma_g),*)$.
 Here $*=\{*_1,\dots,*_n\}$ is a base configuration.
A presentation for this group was first obtained by G. P. Scott \cite{Scott} and revisited by Gonz\'ales-Meneses \cite{Gonzalez}, Bellingeri \cite {Bellingeri}.
 The braid group of a bounded surface has a natural quotient isomorphic to the Heisenberg group of the surface. This is proved in \cite{HeisenbergHomology} for a surface with one boundary component. In the closed case with genus $g>1$ a similar quotient produces a version of the Heisenberg group with finite center; see \cite[Section 5, Example 1]{B_al2022} in case $n\geq 3$.  We will recover the full Heisenberg group by using an $S^1$-bundle over the configuration space. Let us equip $\Sigma_g$ with a riemannian metric (the choice is irrelevant). This determines a conformal structure on $\Sigma_g$, which is equivalent to a complex structure. Then the configuration space $\mathcal{C}_{n}(\Sigma_g )$ inherits a complex structure with hermitian metric and a symplectic structure.

Using the complex structure we may define various  bundles  over the configuration space $\mathcal{C}_{n}(\Sigma_g )$. We have the complex tangent bundle $T_\C(\mathcal{C}_{n}(\Sigma_g ))$, its determinant
 $\Delta(\mathcal{C}_{n}(\Sigma_g ))=\Lambda^n(T_\C(\mathcal{C}_{n}(\Sigma_g ))$, the square determinant
 $\Delta^2(\mathcal{C}_{n}(\Sigma_g ))=\Delta(\mathcal{C}_{n}(\Sigma_g ))^{\otimes 2}$.

\begin{definition}
a) The weakly framed configuration space $\mathcal{C}^\fr_{n}(\Sigma_g )$ of a closed riemannian surface $\Sigma_g$ is the unit bundle in the square determinant $\Delta^2(\mathcal{C}_{n}(\Sigma_g ))$.\\
b) The weakly framed surface braid group is the fundamental group  $\mathbb{B}^\fr_{n}(\Sigma_g)=\pi_{1}(\mathcal{C}^\fr_{n}(\Sigma_g), *^\fr)$. Here $ *^\fr$ is a lift of the base configuration $*$.
\end{definition}

Using the symplectic structure we also have the lagrangian grassmannian bundle $\mathcal{L}(\mathcal{C}_{n}(\Sigma_g ))$, whose fiber is the  grassmanian of lagrangian $n$-spaces in $\C^n=\R^{2n}$ which can be identified with
 $U(n)/O(n)$. We have a square determinant map
\mbox{$ \det^2: \mathcal{L}(\mathcal{C}_{n}(\Sigma_g ))\rightarrow \mathcal{C}^\fr_{n}(\Sigma_g )$},
which allows to consider  $\mathcal{C}^\fr_{n}(\Sigma_g )$ as a quotient of  the lagrangian grassmannian bundle $\mathcal{L}(\mathcal{C}_{n}(\Sigma_g ))$.

The framed braid group of surfaces $FB_n(\Sigma_g)$ is studed in \cite{BellingeriGervais}. It is defined as the fundamental group of the space $F_n(\Sigma_g)$ of $n$-points configurations with a unit tangent vector at each point.
 A framing generates a lagrangian subspace which gives a map
 $FB_n(\Sigma_g)\rightarrow \mathcal{L}(\mathcal{C}_{n}(\Sigma_g ))$. Composing with the projection we obtain a fibration
$FB_n(\Sigma_g)\rightarrow \mathcal{C}^\fr_{n}(\Sigma_g )$ whose fiber is the kernel of $\det^2:(S^1)^n\rightarrow S^1$.
 We can deduce an homomorphism $FB_n(\Sigma_g)\rightarrow \mathbb{B}^\fr_{n}(\Sigma_g)$ whose image is an index $2$ subgroup which identifies each framing generator with $F^2$, where $F$ is the weak framing generator (see below for a definition).
 A presentation of our weakly framed braid group $\mathbb{B}^\fr_{n}(\Sigma_g)$ can be then deduced from \cite[Theorem 13]{BellingeriGervais}.
  We will give below a short proof which will clarify our conventions and choice of generators.

We fix a decomposition of $\Sigma_g$ as a disc with $2g$ handles of index $1$, which gives $\Sigma_{g,1}$, completed by a final handle of index $2$.
The  based loops, $\alpha_1,\dots,\alpha_g,\beta_1,\dots,\beta_g$ are depicted in Figure \ref{modelSurf}.
\begin{figure}
\centering
\includegraphics[scale=0.7]{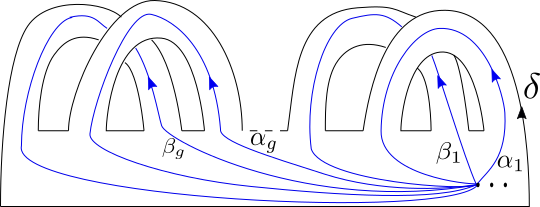}
\caption{Model for $\Sigma_g$; the index $2$ handle is attached along $\delta$ with reversed orientation.}
\label{modelSurf}
\end{figure}
The boundary of the $2$-handle gives a loop homotopic to $\delta^{-1}$ with  $\delta=\beta_g \overline\alpha_g\overline \beta_g\alpha_g \dots \beta_1 \overline\alpha_1\overline \beta_1\alpha_1$, which gives the relation in $\pi_1(\Sigma_g,*_1)$. Here we write the composition of loops from right to left.
   We fix a unit vector field $X$ on $\Sigma_{g,1}$. The loops $\alpha_i$, $\beta_i$, $1\leq i\leq g$, represent free  generators for $\pi_1(\Sigma_{g,1},*_1)$. Here the base point $*_1$ belongs to the base configuration $*$. We will use the same notation $\alpha_i$, $\beta_i$, $1\leq i\leq g$, for the corresponding braids where the weak framing is given by the square determinant of the framing obtained using $X$  at each point in the configuration. We have classical (positive) generators $\sigma_1$, \dots ,$\sigma_{n-1}$ where the weak framing is also given by $X$. We  have a weak framing generator $F$, which rotate counterclockwise the framing vector by $\pi$ around $*_1$. The choice of the vector field $X$ is irrelevant in the presentation, but will be needed when acting with mapping classes.
  For comparing with Bellingeri and al. relations, note that they involve the negative classical generators; see e.g. \cite[Fig. 1]{Bellingeri}.
\begin{theorem}\label{presentation}
For $n\geq 2$, the weakly framed braid group is generated by $\alpha_1,\dots,\alpha_g$, $\beta_1,\dots,\beta_g$, $\sigma_1$, \dots ,$\sigma_{n-1}$ , $F$,
with relations
\[
\begin{cases}
\, \text{$F$ is central},\\
\,\text{(\textbf{BR1}) }\, [\sigma_{i},\sigma_{j}] = 1 & \text{for } \lvert i-j \rvert \geq 2, \\
\,\text{(\textbf{BR2}) }\, \sigma_{i}\sigma_{j}\sigma_{i}=\sigma_{j}\sigma_{i}\sigma_{j} & \text{for } \lvert i-j \rvert = 1, \\
\,\text{(\textbf{CR1}) }\, [\alpha_{r},\sigma_{i}] = [\beta_{r},\sigma_{i}] = 1 & \text{for } i>1 \text{ and all } r, \\
\,\text{(\textbf{CR2}) }\, [\alpha_{r},\sigma_{1}\alpha_{r}\sigma_{1}] = [\beta_{r},\sigma_{1}\beta_{r}\sigma_{1}] = 1 & \text{for all } r, \\
\,\text{(\textbf{CR3}) }\, [\alpha_{r},\sigma^{-1}_{1}\alpha_{s}\sigma_{1}] = [\alpha_{r},\sigma^{-1}_{1}\beta_{s}\sigma_{1}] = & \\
\qquad\qquad\qquad = [\beta_{r},\sigma^{-1}_{1}\alpha_{s}\sigma_{1}] = [\beta_{r},\sigma^{-1}_{1}\beta_{s}\sigma_{1}] = 1 & \text{for all } r<s, \\
\,\text{(\textbf{SCR}) }\, \sigma_{1}\beta_{r}\sigma_{1}\alpha_{r}\sigma_{1}=\alpha_{r}\sigma_{1}\beta_{r} & \text{for all } r,\\
\,\text{(\textbf{FR}) }\, \beta_g \overline\alpha_g\overline \beta_g\alpha_g \dots \beta_1 \overline\alpha_1\overline \beta_1\alpha_1 \, \sigma_1\dots \sigma_{n-1}\sigma_{n-1}\dots \sigma_1= F^{4g-4}, \\

\end{cases}
\]
\end{theorem}
\begin{proof}
The braid group $\mathbb{B}_{n}(\Sigma_g)$ is generated by $\alpha_1,\dots,\alpha_g,\beta_1,\dots,\beta_g$, $\sigma_1$, \dots ,$\sigma_{n-1}$ with relations $BR$, $CR$, $SCR$ and
$$
\,\text{(\textbf{TR}) }\, \beta_g \overline\alpha_g\overline \beta_g\alpha_g \dots \beta_1 \overline\alpha_1\overline \beta_1\alpha_1 \, \sigma_1\dots \sigma_{n-1}\sigma_{n-1}\dots \sigma_1= 1.
$$
This is a slight reformulation of \cite[Theorem 1.2]{Bellingeri}. Weakly framed configurations form an oriented $S^1$-bundle over configurations hence the weakly framed braid group $\mathbb{B}^\mathtt{f}_{n}(\Sigma_g)$  is a central extension of $\mathbb{B}_{n}(\Sigma_g)$. It is generated by lifts of the previous generators and an extra central generator $F$. The relations are obtained from those for $\mathbb{B}_{n}(\Sigma_g)$, by correcting with the appropriate power of $F$. Let $\Sigma_{g,1}$ be the bounded surface obtained before gluing the index $2$ handle. The braid group $\mathbb{B}_{n}(\Sigma_{g,1})$ has presentation with generators $\alpha_1,\dots,\alpha_g,\beta_1,\dots,\beta_g$, $\sigma_1$, \dots ,$\sigma_{n-1}$ and relations $BR$, $CR$, $SCR$. Using a non singular vector field $X$ on $\Sigma_{g,1}$, we obtain homomorphisms
$$\mathbb{B}_{n}(\Sigma_{g,1})\rightarrow \mathbb{B}^{\mathtt{f}}_{n}(\Sigma_{g,1}) \rightarrow \mathbb{B}^{\mathtt{f}}_{n}(\Sigma_{g}) \ .$$
This implies that using this lifts for the generators $\alpha_1,\dots,\alpha_g,\beta_1,\dots,\beta_g$, $\sigma_1$, \dots ,$\sigma_{n-1}$, relations $BR$, $CR$, $SCR$ hold in $\mathbb{B}^{\mathtt{f}}_{n}(\Sigma_{g})$. It remains to check the framing correction for the last relation. The lift of the left hand side in (TR) is represented by the loop $\delta$ with the framing given by $X$. The loop $\delta$ is turning negatively around the outside $2$-cell hence the framing turns $2g-2$ times along $\delta$. The square determinant turns $4g-4$ times which gives relation (FR).
\end{proof}

\section{Heisenberg homologies}
In genus $1$, the braid group $\mathbb{B}_{n}(\Sigma_1)$ quotiented by $\sigma_1$ made central is isomorphic to the standard discrete Heisenberg group, and the construction from \cite{HeisenbergHomology} applies. In this section we will suppose $g>1$ and  consider a version of the discrete Heisenberg group designed for our situation.  The Heisenberg group  $\mathcal{H}_g$ is $\Z\nu \times H_1(\Sigma_g,\Z)$ with $\nu=\frac{\gcd(2g-2,g+n-1)}{2g-2}\in \Q$, and  operation
\begin{equation}
\label{eq:Heisenberg-product}
(k,x)(l,y)=(k+l+\,x.y,x+y).
\end{equation}
 It will be convenient to further embed this group in the rational or real Heisenberg groups $\mathcal{H}_\Q$, $\mathcal{H}_\R$ which motivate a formulation where the center is identified with $\mathbb{Z}\nu$ rather than $\Z$.
 For $g>1$,
 we will use the notation $a_i$, $b_i$ for the homology classes of $\alpha_i$, $\beta_i$, $1\leq i\leq g$.
\begin{proposition}
\label{hom_phi}
a) For each $g> 1$ and $n\geq 2$, there is a surjective  homomorphism
\[
\phi \colon \mathbb{B}^\fr_{n}(\Sigma_g) \relbar\joinrel\twoheadrightarrow \Heis_g
\]
sending each $\sigma_i$ to $u=(1,0)$, $\alpha_i$ to $\tilde{a}_i=(0,a_i)$, $\beta_i$ to $\tilde{b}_i=(0,b_i)$ and $F$ to \mbox{$v=(\frac{g+n-1}{2g-2},0)$}.\\
b) The kernel of $\phi$ is normally generated by the commutators $[\sigma_1,x]$, $x\in \mathbb{B}^\fr_{n}(\Sigma_g)$,
and $\sigma_1^{g+n-1}F^{2-2g}$.
\end{proposition}
\begin{proof}
For the statement a) it is enough to show that all relations in Theorem \ref{presentation} are satisfied in $\mathcal{H}_g$ for the images of the generators. This is straightforward for relations  (BR) and (CR). For (SCR), we get
$$\phi(\sigma_1\beta_s\sigma_1\alpha_s\sigma_1)=(2,a_s+b_s)=\phi(\alpha_s\sigma_1\beta_s).$$
Denote by $lhs$ the left hand side in relation (FR). We have
$$\phi(lhs)=(2g+2n-2,0)=\phi(F^{4g-4}).$$
The subgroup $\mathrm{Im}(\phi)$ is generated by $u=(1,0)$, $v=(\frac{g+n-1}{2g-2},0)$, $\tilde{a}_i=(0,a_i)$, $\tilde{b}_i=(0,b_i)$,  $1\leq i\leq g$,
with $u$ and $v$ central and relations $$u^{g+n-1}=v^{2g-2}\ ,$$
 $$xy=u^{2x.y} yx \text{ for $x,y$ among the $\tilde a_i$,$\tilde b_i$},\ .$$
 This subgroup contains $\nu$ and hence is equal to $\Heis_g$.
From the presentation in Theorem \ref{presentation}, we obtain that the quotient of $\mathbb{B}^\fr_{n}(\Sigma_g)$ by
$[\sigma_1,\mathbb{B}^\fr_{n}(\Sigma_g)]$ and $\sigma_1^{g+n-1}F^{2-2g}$ is generated by $\sigma_1$, $\alpha_1,\dots,\alpha_g,\beta_1,\dots,\beta_g$,  $F$,
with relations
\[
\begin{cases}
\, \text{$F$ and $\sigma_1$ are central},\\
\,\text{(\textbf{CR3}) }\, [\alpha_{r},\alpha_{s}] = [\alpha_{r},\beta_{s}]
  = [\beta_{r},\beta_{s}] = 1 & \text{for all } r<s, \\
\,\text{(\textbf{SCR}) }\, \sigma_{1}^2\beta_{r}\alpha_{r}=\alpha_{r}\beta_{r} & \text{for all } r,\\
\,\text{(\textbf{FR}) }\, \sigma_1^{g+n-1}= F^{2g-2}, \\

\end{cases}
\]
The homomorphism $\phi$ matches the two presentations.
The statement b) for the kernel follows.
\end{proof}
Using the homomorphism $\phi$ we define a regular covering $\widetilde{\mathcal{C}}^\fr_n(\Sigma_g)$ of the weakly framed configuration space ${\mathcal{C}}^\fr_n(\Sigma_g)$.
The homology  of this cover is what we call
 the {\em Heisenberg homology}. Deck transformations endow Heisenberg homology with a right module structure over the group ring $\Z[\Heis_g]$. We may specialise to local coefficients as follows.
Let us denote by $S_*(\widetilde{\mathcal{C}}^\fr_n(\Sigma_g))$ the singular  chain complex of the Heisenberg cover, which is a right $\Z[\Heis_g]$-module. Given a representation $\rho: \Heis_g \rightarrow GL(V)$, the corresponding local homology is that of the complex $S_*(\mathcal{C}^\fr_n(\Sigma),V):=S_*(\widetilde{\mathcal{C}}^\fr_n(\Sigma))\otimes_{\Z[\Heis_g]} V$.
It will be called the Heisenberg homology of weakly framed surface configurations with coefficients in $V$.

It is convenient  to also consider  Borel-Moore homology
\[
H_*^{BM}(\mathcal{C}^\fr_{n}(\Sigma_g);V) =
{\varprojlim_T}\, H_*(\mathcal{C}^\fr_{n}(\Sigma_g), \mathcal{C}^\fr_{n}(\Sigma_g)\setminus T^\fr ; V),
\]
the inverse limit is taken over all compact subsets  $T\subset\mathcal{C}_{n}(\Sigma_g)$, and $T^\fr\subset \mathcal{C}^\fr_{n}(\Sigma_g)$ denotes the corresponding weakly framed configurations.

\section{Action of the $\fr$-based Mapping Class Group}
Recall that $*=\{*_1,\dots,*_n\}\in  \mathcal{C}_{n}(\Sigma_g )$, $g\geq 2$, $n\geq 2$, is the base $n$-points configuration. We denote by $\mathfrak{M}(\Sigma_g,*)$ the Mapping Class Group of the punctured surface.
 By a theorem of Moser \cite{Moser}, we may work with representatives of mapping classes which are area preserving, equivalently in this dimension with symplectomorphisms.
Let us denote by $\mathcal{C}_n(f)$ the  diffeomorphism  of  $\mathcal{C}_n(\Sigma_g)$ corresponding to a symplectomorphism $f$ which fixes the base configuration.
It is a symplectomorphism giving an action on the lagrangian bundle $\mathcal{L}(\mathcal{C}_{n}(\Sigma_g ))$. We obtain an induced action $\mathcal{C}^\fr_n(f)$ on the square determinant quotient $\mathcal{C}^\fr_n(\Sigma_g)$. We denote by $*^\fr$ the base configuration with a choice of weak framing, i.e. an inverse image of $*$ in $\mathcal{C}^\fr_n(\Sigma_g)$. We will consider here an extension of the Mapping Class Group obtained with isotopy classes of symplectomorphism fixing the base configuration with weak framing $*^\fr$.

We fix a lift $\tilde *^\fr\in \widetilde{\mathcal{C}}^\fr_n(\Sigma_g)$ of the weakly framed base configuration $*^\fr$.
\begin{proposition}
\label{f_lift}
Let $f$ be a symplectomorphism fixing $*^\fr$, then  $\mathcal{C}^\mathtt{f}_n(f)$ lifts  uniquely to a diffeomorphism
 $$\widetilde{\mathcal{C}}^\fr_n(f):\widetilde{\mathcal{C}}^\fr_n(\Sigma_g)\rightarrow \widetilde{\mathcal{C}}^\fr_n(\Sigma_g)\ ,$$
which fixes $\tilde *^\fr$.
 \end{proposition}

\begin{proof}
The diffeomorphism $\mathcal{C}^\fr_n(f)$ fixes the base point $*^\fr\in \mathcal{C}^\fr_n(\Sigma_g)$. It induces an automorphism $\mathcal{C}^\fr_n(f)_\sharp$ of  $\mathbb{B}^\fr_n(\Sigma_g)=\pi_1(\mathcal{C}^\fr_n(\Sigma_g), *^\fr)$
 which fixes the classical generator $\sigma_1$ and the framing generator $F$. Recall that $\widetilde{\mathcal{C}}^\fr_n(\Sigma_g)$ is the regular covering space associated with $\phi: \mathbb{B}^\fr_n(\Sigma_g)\rightarrow \Heis_g$.
 From Proposition \ref{hom_phi} we get that $$\mathcal{C}^\fr_n(f)_\sharp(\mathrm{Ker}(\phi))=\mathrm{Ker}(\phi)\ ,$$
which proves the statement.
\end{proof}
The above argument also proves the following, which can be seen as an extension of similar results on surface braid groups
\cite{An,BGG2017,HeisenbergHomology}.
\begin{proposition}
 There exists a unique automorphism $f_{\Heis} \colon \Heis_g \rightarrow \Heis_g$, which is identity on the center and such that the following square commutes:
\begin{equation}
\label{eq:projection-equivariance}
  \begin{tikzcd}
     \mathbb{B}^\fr_{n}(\Sigma_g) \arrow[d,swap, "\phi"]\arrow[r, "{\mathbb{B}^\fr_{n}(\Sigma)}"] & \mathbb{B}^\fr_{n}(\Sigma_g) \arrow[d,"\phi"]\\
     \Heis_g \arrow[r, "f_{\Heis}"]& \Heis_g
  \end{tikzcd}
\end{equation}
\end{proposition}
\begin{definition}
The $\fr$-based Mapping Class Group $\mathfrak{M}^\fr(\Sigma_g,*^\fr)$ is the group of isotopy classes of symplectomorphisms fixing the base weakly framed configuration $*^\fr$.
\end{definition}

\begin{proposition}
The group $\mathfrak{M}^\fr(\Sigma_g,*^\fr)$ is a $\Z$ central extension of $\mathfrak{M}(\Sigma_g,*)$, with kernel generated by the half twist around the base point $*_1$.
\end{proposition}
\begin{proof}
We have an evaluation map from the group of symplectomorphisms fixing the base configuration $*$ to the fiber $S^1$ over $*$ in the bundle $\mathcal{C}^\fr_n(\Sigma_g)$. This is a fibration and we obtained an exact sequence

\begin{align*}
0\rightarrow \Z=\pi_1(S^1)&\rightarrow \mathfrak{M}^\fr(\Sigma_g,*^\fr)=\pi_0(\mathrm{Symp}^\fr(\Sigma_g,*^\fr))\\ & \rightarrow\mathfrak{M}(\Sigma_g,*)=\pi_0(\mathrm{Symp}(\Sigma_g,*))\rightarrow 1
\end{align*}

The isotopy between the identity and the half twist around $*_1$ rotates the framing at $*_1$ by $\pi$, which generates $\pi_1$ of the fiber. This identifies the kernel generator. This half twist commutes with symplectomorphisms
which are identity on a disc neighbourhood of the base configuration $*$. One can check that it also commutes up to isotopy fixing $*^\fr$ with symplectomorphisms supported in a disc containing $*$, which are classical braids. Composing with classical braids any symplectomorphism fixing $*^\fr$ is isotopic to one which is the identity on a disc neighbourhood of $*$. Centrality follows.
\end{proof}

We denote by $\mathrm{Aut}^+(\Heis_g)$ the group of {\em oriented} automorphisms of $\Heis_g$ which means automorphisms which are identity on the center.
We have an action of the $\fr$-based Mapping Class Group on the Heisenberg group $\Heis_g$,
$\Psi \colon  \mathfrak{M}^\fr(\Sigma_g,*^\fr) \longrightarrow \mathrm{Aut}^+(\Heis_g)$, $f\mapsto f_\Heis$.
The quotient of $\Heis_g$ by its center is equal to $H_1(\Sigma_g,\Z)$ hence every  oriented automorphism  $\tau\in \mathrm{Aut}^+(\Heis_g)$ induces an automorhism of $H_1(\Sigma_g,\Z)$. The triviality of the action on the center implies that the induced map
$\overline{\tau}$  is symplectic, so we have an homomorphism $\mathrm{Aut}^+(\Heis_g) \to Sp(H_1(\Sigma_g,\Z))$.  This homomorphism has a section and its kernel is isomorphic to $\mathrm{Hom}(H_1(\Sigma_g,\Z),\Z\nu)\cong H^1(\Sigma_g,\Z\nu)$, see \cite[Lemma 16]{HeisenbergHomology}. This identifies the group $\mathrm{Aut}^+(\Heis_g)$ as a semidirect product
$$\mathrm{Aut}^+(\Heis_g) \cong Sp(H_1(\Sigma_g)) \ltimes H^1(\Sigma_g,\Z\nu)\ .$$
The action of a $\fr$-based mapping class $f\in \mathfrak{M}^\fr(\Sigma_g,*^\fr)$ on $\Heis_g$ writes down
$$f_\Heis: (k,x)\mapsto (k+\delta_f(x),f_*(x))\ ,$$
where $\delta: \mathfrak{M}^\fr(\Sigma_g,*^\fr) \rightarrow H^1(\Sigma_g,{\Z}\nu)$ is a crossed homomorphism, i.e.
for all $f,g\in \mathfrak{M}(\Sigma_g)$ we have
$\delta_{g\circ f}=\delta_f+f^*(\delta_g)$, see \cite[Section 3.3]{HeisenbergHomology}.

From proposition \ref{f_lift} we obtain for $f\in \mathfrak{M}^\fr(\Sigma_g,*^\fr)$ an homology isomorphism
\begin{equation*}
\label{iso}
\widetilde{\mathcal{C}}^\fr_n(f)_* \colon H_*( \mathcal{C}^\fr_n(\Sigma_g),\Z) \longrightarrow H_*\bigl( \mathcal{C}^\fr_n(\Sigma_g) ,\Z\bigr)
\end{equation*}
which is $\Z$-linear. This provides a representation of the $\fr$-based Mapping Class Group
$$\mathfrak{M}^\fr(\Sigma_g,*^\fr)\rightarrow \mathrm{Aut}_\Z(  H_*( \mathcal{C}^\fr_n(\Sigma_g),\Z))\ .$$
This representation is twisted with respect to the right $\Z[\Heis_g]$-module structure, which means
that for $x\in H_*( \mathcal{C}^\fr_n(\Sigma_g),\Z)$, $h\in \Heis_g$, we have
$$\widetilde{\mathcal{C}}^\fr_n(f)_*(x.h)=\widetilde{\mathcal{C}}^\fr_n(f)_*(x).f_\Heis(h)\ .$$

For a representation $\rho: \Heis_g\rightarrow GL(V)$  and automorphism $\tau\in \mathrm{Aut}^+(\Heis_g)$, we denote by ${}_{\tau}\!V$ the twisted representation $\rho\circ \tau$.
Recall that the homology with coefficient in $V$ is computed from the complex $S_*(\mathcal{C}^\fr_n(\Sigma),V):=S_*(\widetilde{\mathcal{C}}^\fr_n(\Sigma_g))\otimes_{\Z[\Heis_g]} V$.
\begin{theorem}
There is a natural twisted representation of the $\fr$-based Mapping Class Group $\mathfrak{M}^\fr(\Sigma_g,*^\fr)$ on
\begin{equation*}
H_*\bigl( \mathcal{C}^\fr_n(\Sigma_g),{}_\tau\!V \bigr) \ , \quad \tau \in \mathrm{Aut}^+(\Heis_g)\ ,\
\end{equation*}
where the action of $f\in \mathfrak{M}^\fr(\Sigma,*^\fr)$ is
\begin{equation*}
\label{twisted}
\mathcal{C}^\fr_n(f)_* \colon H_*\bigl( \mathcal{C}^\fr_n(\Sigma_g) , {}_{\tau\circ f_\Heis}\!V \bigr) \longrightarrow H_*\bigl( \mathcal{C}^\fr_n(\Sigma_g) ,{}_{\tau}\!V \bigr)
\end{equation*}
\end{theorem}
\begin{proof}
The action of $\widetilde{\mathcal{C}}^\fr_n(f)$ on chains is twisted with respect to the $\Z[\Heis_g]$ action, which writes down
$$S_*(\widetilde{\mathcal{C}}^\fr_n(f)(zh)=S_*(\widetilde{\mathcal{C}}^\fr_n(f)(z)f_\Heis(h), \text{ for $z\in S_*(\widetilde{\mathcal{C}}^\fr_n(\Sigma_g))$, $h\in \Heis_g$.}$$
We check that the map $z\otimes v\mapsto S_*(\widetilde{\mathcal{C}}^\fr_n(f)(z)\otimes v$ induces an isomorphism
$$S_*(\mathcal{C}^\fr_n(\Sigma),{}_{\tau\circ f_\Heis}\!V)\rightarrow S_*(\mathcal{C}^\fr_n(\Sigma),{}_{\tau}\!V)\ ,$$
which produces the functorial twisted action on the homologies.
\end{proof}
\section{MCG representations from the regular action on Heisenberg group}\label{regular}
In this section we obtain finite dimensional representations of the Mapping Class Groups from the left regular action on the Heisenberg group $\Heis_g^\Q$. The group $\Heis_g$ is a subgroup in $\Heis_g^\Q$.
We endow $\Heis_g^\Q=\Q\times H_1(\Sigma_g,\Q)$ with affine structure isomorphic to $\Q^{2g+1}$.
The left regular action $l_{(k_0,x_0)}$ is then an affine automorphism. We decompose $x_0=p_0+q_0$, $p_0\in \Lambda_a=\mathrm{Span}(a_i, 1\leq i\leq g)$, $q_0\in \Lambda_b=\mathrm{Span}(b_i, 1\leq i\leq g)$, then the action  is written
$$\begin{cases}
k'=k+k_0+p_0.q-q_0.p\\
p'=p+p_0\\
q'=q+q_0
\end{cases}$$
We  consider the linearisation $\rho_L$ of this affine action on  $L=\Heis_g^\Q\oplus \Q$. The linear action of
$\rho_L(k_0,x_0)$
is as follows.
$$\begin{cases}
k'=k+t k_0+p_0.q-q_0.p\\
p'=p+t p_0\\
q'=q+t q_0\\
t'=t
\end{cases}$$
The nice feature of this representation is that the twisted representation ${}_{\tau}\!L$  is canonically isomorphic to $L$.
\begin{lemma}
For $\tau \in \mathrm{Aut}^+(\Heis_g)$,  the linear map $\tau\times \id_\Q: L\mapsto {}_{\tau}\!L$ gives an isomorphism of
$\Z[\Heis_g]$-module.
\end{lemma}
\begin{proof}
We first check that $\tau$ intertwines the affine action $l_{(k_0,x_0)}$ on $\Heis_g^\Q$ and the twisted affine action
$l_{\tau(k_0,x_0)}$.
We have
$$l_{\tau(k_0,x_0)})(k,x)=\tau(k_0,x_0)(k,x)=\\
\tau\left((k_0,x_0)\tau^{-1}(k,x)\right)=\tau\left(l_{(k_0,x_0)}(\tau^{-1}(k,x)\right)\ .$$
The result is written
$$ l_{\tau(k_0,x_0)}=\tau\circ l_{(k_0,x_0)}\circ \tau^{-1}\ .$$
After linearisation we obtain the intertwinning formula
$$ (\rho_L\circ\tau)(k_0,x_0))=(\tau\times \id_\Q)\circ \rho_L(k_0,x_0)\circ (\tau^{-1}\times \id_\Q)\ .$$
\end{proof}
Composing the homology isomorphism induced by the intertwinning of representations with the twisted action from Theorem \ref{twisted}, we obtain a natural homological action  of $\fr$-based mapping classes by automorphisms.
\begin{theorem}There is  a representation
$$\mathfrak{M}^\fr(\Sigma_g)\rightarrow \mathrm{Aut}(H_*\bigl( \mathcal{C}^\fr_n(\Sigma_g) ,L\bigr)\ ,$$
which associates to $f\in  \mathfrak{M}(\Sigma_g)$
the composition of the coefficient isomorphism  induced by $f_\Heis$, $$H_*\bigl( \mathcal{C}^\fr_n(\Sigma_g),L \bigr)\cong
H_*\bigl( \mathcal{C}^\fr_n(\Sigma_g) ,{}_{f_\Heis}\!L \bigr)\ ,$$
 with the functorial homology isomorphism
$$\mathcal{C}^\fr_n(f)_* \colon H_*\bigl( \mathcal{C}^\fr_n(\Sigma_g) , {}_{f_\Heis}\!L\bigr) \longrightarrow H_*\bigl( \mathcal{C}^\fr_n(\Sigma_g) ,L \bigr)\ .$$
\end{theorem}
\section{About computation}
In \cite{HeisenbergHomology} it is proved that a relative Borel-Moore Heisenberg homology of configurations in $\Sigma_{g,1}$ is free of finite dimension over the group ring of the Heisenberg group. The argument does not work for closed surfaces. A more careful analysis of a cell decomposition of weakly framed configurations is likely to be needed.
We first quote that $\mathcal{C}^\fr_n(\Sigma_g)$ has the homotopy type of a finite CW-complex. Indeed,  we get the same homotopy type if we consider weakly framed configurations where points cannot be $\epsilon$-closed with $\epsilon$ small enough, i.e we replace the condition $x_i\neq x_j$ by $d(x_i,x_j)\geq \epsilon$, $i\neq j$, and get a compact manifold with boundary $\mathcal{C}^{\epsilon,\fr}_n(\Sigma_g)$. It follows that for finite dimensional representations of the Heisenberg group, the obtained homologies are finite dimensional.

It is exciting to analyse submanifolds representing cycles in Heisenberg homologies, expecting that certain family could generate a subspace invariant under Mapping Class Group action. Let us denote by  $\mathcal{C}^\fr_{*_1,n-1}(\Sigma_g)$ the subspace of weakly framed configurations containing the point $*_1$. For a partition $n=n_1+n_2+\dots+n_{2g-1}+n_{2g}$, we obtain a cell formed with configurations having $n_1$, $n_2$, \dots,$n_{2g-1}$,  $n_{2g}$ points respectively on $\alpha_1$, $\beta_1$, \dots, $\alpha_g$, $\beta_g$,
weakly framed by the vector field $X$. This gives a properly embedded cell representing an homology class in $H^{BM}_*\bigl( \mathcal{C}^\fr_n(\Sigma_g) ,\mathcal{C}^\fr_{*_1,n-1}(\Sigma_g),V \bigr)$, where $\mathcal{C}^\fr_{*_1,n-1}(\Sigma_g)$
denotes the subspace of $n$-points configurations containing $*_1$. We will obtain classes in $H_n\bigl( \mathcal{C}^\fr_n(\Sigma_g) ,V)$ by studying the kernel of the boundary map $$H^{BM}_n\bigl( \mathcal{C}^\fr_n(\Sigma_g) ,\mathcal{C}^\fr_{*_1,n-1}(\Sigma_g);V \bigr)\rightarrow
H^{BM}_{n-1}\bigl( \mathcal{C}^\fr_{*_1,n-1}(\Sigma_g),V \bigr)\ .$$ The case $n=2$ already looks promising.

\bibliographystyle{plain}
\bibliography{biblio}

\end{document}